\documentclass[11pt,reqno]{amsart}
\usepackage[letterpaper,margin=1in,footskip=0.25in]{geometry}
\usepackage{amssymb}
\usepackage{mathrsfs}
\usepackage[bookmarks, bookmarksdepth=2, colorlinks=true, linkcolor=blue, citecolor=blue, urlcolor=blue]{hyperref}
\usepackage{mathtools}
\usepackage{array}
\usepackage{diagbox}

\usepackage{tikz,lipsum}
\usetikzlibrary{matrix}

\newcommand{\rvline}{\hspace*{-\arraycolsep}\vline\hspace*{-\arraycolsep}}

\newtheorem{theorem}{Theorem}[section]
\newtheorem{lemma}[theorem]{Lemma}

\newtheorem{proposition}[theorem]{Proposition}

\newtheorem{conjecture}[theorem]{Conjecture}

\newtheorem{remark}[theorem]{Remark}
\newtheorem{definition}[theorem]{Definition}
\newtheorem{example}[theorem]{Example}
\newcommand{\N}{\mathbf{N}}

\newcommand{\fd}{f_{\delta}}
\newcommand{\fx}{f_{\delta+x}}
\newcommand{\fy}{f_{\delta+y}}
\newcommand{\fxy}{f_{\delta+x+y}}

\newcommand{\ird}{I_r^{(d)}}
\newcommand{\sinfty}{\mathfrak{S}_\infty}
\newcommand{\fs}{\mathfrak{S}}

\newcommand{\fp}{\mathfrak{p}}
\newcommand{\fJ}{\mathfrak{J}}

\begin{document}
\title{On the non primality of certain symmetric ideals}
\author{Hyung Kyu Jun}
\address{Department of Mathematics\\ University of Michigan\\Ann Arbor, MI
48103\\USA}
\email{\href{mailto:hkjun@umich.edu}{hkjun@umich.edu}}



\makeatletter
  \hypersetup{
    pdfsubject=\@subjclass,pdfkeywords=\@keywords
  }
\makeatother

\begin{abstract} Let $R =k[x_1, \cdots, x_n,\cdots]$ be the infinite variable polynomial ring equipped with the natural $\sinfty$ action, where $k$ is a field of characteristic zero. In recent work \cite{NS21}, Nagpal--Snowden gave an indirect proof that $\sinfty$-ideal generated by $(x_1-x_2)^{2n}$ is not $\sinfty$-prime. In this paper, we give a direct proof, with explicit elements. We further formulate some conjectures on possible generalizations of the result.
\end{abstract}

\maketitle

\setcounter{tocdepth}{1}
{\hypersetup{hidelinks}\tableofcontents}

\section{Introduction}\label{sect:intro}
\subsection{Background}
Let $k$ be a field of characteristic zero, and let $R = k[x_1, x_2, \cdots]$ be the polynomial ring in countably many variables $x_i$. The ring $R$ is equipped with a natural action of the  infinite symmetric group $\fs= \bigcup_{n \in \N} \mathfrak{S}_n$, where $\mathfrak{S}_n$ is the symmetric group of degree $n$. The ring $R$ is not noetherian, but the size of $\fs$ bridges the gap: in \cite{Co67} Cohen showed that $R$ is equivariantly noetherian with respect to $\fs$. Equivariant noetherianity of $R$ leads us to consider $\fs$-equivariant commutative algebra on $R$. One is interested in translating notions from the familiar notions from noetherian rings to a non-noetherian ring $R$, equipped with an action of $\fs$ (see \cite{NS21} for basic notions of $G$-equivariant commutative algebra). As prime ideals are one of the most important notions in commutative algebra and algebraic geometry, we are interested in $\fs$-prime ideals:

\begin{definition}\label{def:sprime} An $\mathbf{\fs}$-\textbf{ideal} of $R$ is an ideal closed under the action of $\fs$. An $\fs$-ideal $\fp$ is $\mathbf{\fs}$-\textbf{prime} if $f \cdot \sigma(g) \in \fp$ for all $\sigma \in \fs$ implies $f \in \fp$ or $g \in \fp$. 
\end{definition}
In fact, $\fs$-prime ideals of $R$ are completely classified in \cite{NS21}, and one of the main steps from the classification of $\fs$-primes in \cite{NS21} is the following: 
\begin{theorem}\label{theorem:NS}  For any positive integer $N$, the ideal $ I(N):= \langle (x_i-x_j)^{N} \rangle$ is $\fs$-prime if and only if $N$ is odd.

\end{theorem}

Although Theorem \ref{theorem:NS} tells us that $I(2n)$ is not $\fs$-prime, the proof in \cite{NS21} does not provide $f,g$ such that $f \sigma(g) \in I(2n)$ for all $\sigma$. The goal of this paper is to provide such an explicit pair. 

\subsection{Main Results} The goal of this paper is to prove the following theorem:

\begin{theorem}\label{thm:nonprimality} If $f = g = (x_1-x_2)^{2n-1}$, then for any $\sigma \in \fs$, we have $f \cdot \sigma(g) \in I(2n)$.
\end{theorem}
Theorem \ref{thm:nonprimality} implies that $I(2n)$ is not $\fs$-prime for any $n$ as $(x_1-x_2)^{2n-1} \not\in I(2n)$, because any nonzero element of a homogeneous ideal $I(2n)$ is of degree $\geq 2n$. 

For $n=1$, we can directly show that Theorem \ref{thm:nonprimality} holds:

\begin{example} Let $f=g= x_1-x_2,$ and $\sigma\in \fs$ be any permutation. Then, as any nonzero element of $I(2)$ has degree at least $2$, we see that $f =g\not\in I(2)$. However, note that $f\cdot \sigma(g) \in I(2)$ from the following identity:
\begin{equation}\label{n=1}
    (X-Z)^2 -(Y-Z)^2-(X-W)^2+(Y-W)^2=-2(X-Y)(Z-W).
\end{equation}
If $\sigma$ maps $x_1$ to $x_i$ and $x_2$ to $x_j$, set $X = x_1$, $Y = x_2$, $Z = x_i$, and $W = x_j$, and we see that $f\sigma(g) \in I(2)$ from (\ref{n=1}).
\end{example}

\subsection{Open problems}
In \S\ref{sect:proof}, we will see that Theorem \ref{thm:nonprimality} is related to a contraction of an ideal $I$ in $k[\delta, x,y]$ to $k[x,y]$. See Theorem \ref{theorem:generalization}. We then have two conjectures that generalize Theorem~\ref{theorem:generalization}: Conjecture~\ref{Conj:contracted ideal} explicitly gives monomial generators of the contracted ideal, and Conjecture \ref{Conj:more variables} generalizes Theorem \ref{theorem:generalization} for more variables. See \S\ref{sect:openproblems} for details. 

\subsection{Outline}
In \S\ref{sect:proof}, we give a proof of Theorem \ref{thm:nonprimality}. In \S \ref{sect:openproblems} we discuss about conjectures that generalize our main results.

\subsection*{Acknowledgments}

I am grateful to my advisor Andrew Snowden for his support and for numerous fruitful conversations and suggestions, to Rohit Nagpal for sharing his thoughts on Schur polynomials and Grobner methods, and David E Speyer for insights on morphisms involviong triangular matrices. 

\section{Preliminary linear algebra}

In this section, we will discuss two results that will play key roles in the proof of the Theorem \ref{thm:nonprimality}.

\subsection{Determinant of a matrix with binomial entries}

\begin{proposition} \label{prop:minormatrix} Let $A$ be the following lower-triangular unipotent $N \times N$ matrix:
\begin{equation} A = 
    \begin{bmatrix}
    1 & 0 & \cdots &0 & 0\\ \\
    {N \choose 1} & 1 &\cdots &0 & 0\\ \\
    \vdots &\vdots &\ddots  &\vdots & \vdots \\ \\
    {N \choose N-2} & {N \choose N-3}& \cdots & 1& 0\\ \\
    {N \choose N-1} & {N \choose N-2} &\cdots & {N \choose 1} & 1
    \end{bmatrix}.
\end{equation}
For $1 \leq m \leq N$, let $A_m$ be the following lower-left submatrix of $A$
\begin{equation} A_m = 
    \begin{bmatrix}
    {N \choose N - m} & {N \choose N - m - 1} &\cdots &0 &0\\ \\
    \vdots &\vdots &\ddots &\vdots &\vdots\\ \\
    {N \choose N-2} & {N \choose N-3}& \cdots &{N \choose N - m - 2}& {N \choose N - m - 1}\\ \\
    {N \choose N-1} & {N \choose N-2} &\cdots &{N \choose N - m - 1} & {N \choose N - m}
    \end{bmatrix}.
\end{equation}
Then, $A_m$ is an invertible matrix. 
\end{proposition}

\begin{proof} Let $A_k$ be the $k\times k$ submatrix of the matrix above.  
Note that ${n \choose m}= e_m(1, \cdots, 1)$ where $e_m$ is the elementary symmetric polynomial 

\begin{equation*}
    e_m(x_1, \cdots, x_n) := \sum_{1 \leq i_1 \leq \cdots \leq i_m\leq n}X_{i_1} \cdots X_{i_m}.
\end{equation*}
Then, we see that the determinant of $A_k$ equals the determinant of the $k \times k$ matrix $B = (b_{i,j})$, evaluated at $(1, \cdots, 1)$, where all the entries $b_{i,j}$ are the elementary symmetric polynomials, corresponding to $(i,j)^{th}$ entry of $A_k$. Using Jacobi--Trudi identity (see \cite[pp.~40--41]{Ma98} or \cite[pp.455]{FH13}), we see that $\det(B)=s$, where $s$ is a Schur polynomial. Then the value of $s$ at $(1, \cdots, 1)$ is gives the dimension of the irreducible representation of $GL_n$ with highest weight $\lambda$: that is, if $ \lambda = (\lambda_1\geq \cdots\geq \lambda_k\geq 0)$ and $s= s_\lambda$ is a Schur polynomial associated to the partition $\lambda$, we have the following closed-form expression
\begin{equation*}
    s_\lambda(1,\cdots, 1) = \prod_{i<j}\frac{\lambda_i-\lambda_j+j-i}{j-i} >0. \qedhere
\end{equation*}
\end{proof}

\begin{remark} The above proposition also applies to the matrix whose $(i,j)$th entry is ${n \choose m+i-j}$. The same logic as above applies, so $\det(A)$ is a value of the corresponding Schur polynomial at $(1, \cdots, 1)$. Hence, $\det(A)$ is positive and in particular, nonzero. 
\end{remark}

\subsection{Anti-diagonal transposition}

For $n \times n$ matrix $A$, we define $A^{\tau}$ to be an anti-diagonal transposition,
\[A^{\tau} := JA^TJ,\]
where $J$ is the following matrix: 
\[
J = 
\begin{bmatrix}
0 & 0 & 0 & \cdots & 0 & 0 & 1\\ \\
0 & 0 & 0 & \cdots & 0 & 1 & 0 \\ \\ 
\vdots & \vdots & \vdots & \ddots & \vdots & \vdots & \vdots \\ \\ 
0 & 1 & 0 & \cdots & 0 & 0 & 0 \\ \\
1 & 0 & 0 & \cdots & 0 & 0 & 0
\end{bmatrix}.
\]
Note that for any two matrices $A$ and $B$, we have
\[(A + B)^{\tau} = A^{\tau} + B^{\tau},\]
and 
\[ (AB)^{\tau} = B^{\tau}A^{\tau}. \]
 
We need the following lemma:
\begin{lemma}\label{anti-diagonal} If $A = (a_{i,j})_{1\leq i, j \leq n}, B = (b_{i,j})_{1\leq i, j \leq n} $ are $n \times n$ matrices satisfying:

\begin{itemize}
    \item For any $1 \leq m \leq n$, $m \times m $ ``upper-right'' submatrix\[A_m = 
    \begin{bmatrix}
    a_{1, n - m + 1} & a_{1, n - m + 2} & \cdots & a_{1, n} \\
    \vdots & \vdots & \ddots & \vdots \\
    a_{m, n - m + 1} & a_{m, n + m + 2} & \cdots & a_{m, n}
    \end{bmatrix}\] of $A$ is nonsingular
    \item All the anti-diagonal entries of $B$ are zero
    \item $B^{\tau} = -B$.
\end{itemize}
Then there exists a unique strictly lower-triangular $n \times n$ matrix $X$ satisfying $AX - (AX)^{\tau} = B.$
\end{lemma}

\begin{proof}
Since $X$ is strictly lower triangular matrix, write 
\[
X = 
\begin{bmatrix}
 \big{|}  & \big{|} & \cdots &\big{|} &\big{|} \\
X_1 & X_2 & \ddots & X_{n -1} & X_n \\
\big{|}  & \big{|} & \cdots &\big{|} &\big{|}
\end{bmatrix}, 
\]
where 
\[X_i = 
\begin{bmatrix}
0\\
\vdots\\
0\\
x_{i + 1,i}\\
\vdots\\
x_{n, i}
\end{bmatrix}
\] is a  column vector with $i$ zeros from the top. Specifically, $X_n = 0$. Also write
\[
A =
\begin{bmatrix}
\text{---} \hspace{-0.2cm} & A_1 & \text{---} \hspace{-0.2cm} \\
\text{---} \hspace{-0.2cm} & A_2 &\text{---} \hspace{-0.2cm} \\
\vdots & \ddots & \vdots \\
\text{---} \hspace{-0.2cm} & A_n & \text{---} \hspace{-0.2cm}
\end{bmatrix},
\]
where $A_j$'s are horizontal $1 \times n$ vectors. Then, we see that
\[
AX = 
\begin{bmatrix}
A_1X_1 & A_1X_2 & \cdots & A_1X_{n-1} & 0 \\
A_2X_1 & A_2X_2 & \cdots & A_2X_{n-1} & 0 \\ 
\vdots & \vdots & \ddots & \vdots &\vdots \\
A_{n-1}X_1 & A_{n-1}X_2 & \cdots & A_{n-1}X_{n-1} & 0\\
A_nX_1 & A_nX_2 & \cdots & A_nX_{n-1} & 0
\end{bmatrix}, 
\]
and 
\[
(AX)^{\tau} = 
\begin{bmatrix}
0 & 0 & \cdots & 0 & 0 \\
A_{n}X_{n-1} & A_{n-1}X_{n-1} & \cdots & A_2X_{n-1} & A_1X_{n-1} \\ 
\vdots & \vdots & \ddots & \vdots &\vdots \\
A_nX_2& A_{n-1}X_2 & \cdots & A_2X_{2} & A_1X_2\\
A_nX_1 & A_{n-1}X_1 & \cdots & A_2X_{1} & A_1X_1
\end{bmatrix}, 
\]
Hence, we see that 
\[
AX - (AX)^\tau = 
\begin{bmatrix}
A_1X_1 & A_1X_2 & \cdots & A_1X_{n-1} & 0 \\
A_2X_1 - A_nX_{n - 1} & A_2X_2 - A_{n-1}X_{n-1} & \cdots & 0 & - A_1X_{n-1} \\
\vdots  & \vdots & \ddots & \vdots & \vdots \\
A_{n-1}X_1 - A_nX_2 & 0 & \cdots & - (A_2X_2 -A_{n-1}X_{n-1}) & -A_1X_2 \\
0 & - (A_{n-1}X_1 - A_nX_2) & \cdots & -(A_2X_1 - A_n X_{n-1}) & -A_1X_1
\end{bmatrix}.
\]
Thus, if we let
\[B = 
\begin{bmatrix}
b_{1,1} & b_{1, 2} & \cdots & b_{n-1, n} & 0 \\
b_{2,2} & b_{2, 2} & \cdots & 0 & -b_{1,1} \\
\vdots & \vdots & \ddots & \vdots & \vdots \\
b_{n-1, 1} & 0 & \cdots & -b_{2,2} & - b_{1,2} \\
0 & -b_{n-1, 1} & \cdots &-b_{1,2} & -b_{1,1}
\end{bmatrix}\]
we only need to check that there exists $X_1 \cdots X_{n-1}$ so that strictly upper anti-diagonal entries of $C:= AX - (AX)^\tau$ and those of $B$ match. 

If we look at $i$th columm($0 < i < n$) of $C$ and that of $B$, we require that
\[
\begin{bmatrix}
\text{---} \hspace{-0.2cm} & A_1 & \text{---} \hspace{-0.2cm} \\
\vdots & \vdots & \vdots \\
\text{---} \hspace{-0.2cm} & A_{n-i} & \text{---} \hspace{-0.2cm}
\end{bmatrix} 
\begin{bmatrix}
\vert \\
X_i \\
\vert
\end{bmatrix} 
-
\begin{bmatrix}
\text{---} \hspace{-0.2cm} & 0 & \text{---} \hspace{-0.2cm} \\
\vdots & \vdots & \vdots \\
\text{---} \hspace{-0.2cm} & 0 & \text{---} \hspace{-0.2cm}\\

\text{---} \hspace{-0.2cm} & A_{n - i + 1} & \text{---} \hspace{-0.2cm}
\end{bmatrix} 
\begin{bmatrix}
\vert \\
X_{i + 1} \\
\vert
\end{bmatrix} 
+ \cdots + 
\begin{bmatrix}
\text{---} \hspace{-0.2cm} & A_{n - i + 1} & \text{---} \hspace{-0.2cm} \\
\text{---} \hspace{-0.2cm} & 0 & \text{---} \hspace{-0.2cm}\\
\vdots & \vdots & \vdots \\
\text{---} \hspace{-0.2cm} & 0 & \text{---} \hspace{-0.2cm}
\end{bmatrix} 
\begin{bmatrix}
\vert \\
X_{n - 1} \\
\vert
\end{bmatrix} 
=
\begin{bmatrix}
b_{1,i}\\
b_{2, i}\\
\vdots\\
b_{n - i, i}
\end{bmatrix}
\]

Here, recall that $X_i$ is a colum vector with $i$ consecutive zeros from the top. Hence, 
\[
\begin{bmatrix}
\text{---} \hspace{-0.2cm} & A_1 & \text{---} \hspace{-0.2cm} \\
\vdots & \vdots & \vdots \\
\text{---} \hspace{-0.2cm} & A_{n-i} & \text{---} \hspace{-0.2cm}
\end{bmatrix} 
\begin{bmatrix}
\vert \\
X_i \\
\vert
\end{bmatrix} = 
\begin{bmatrix}
  \bold{0} & \rvline & A_{i},
\end{bmatrix}
\begin{bmatrix}
\vert \\
X_i \\
\vert
\end{bmatrix}
\]
where $A_i$ is $i \times i$ upper right submatrix of $A$. Hence, if we write $X_i'$ to be $(n - i) \times 1$ column vector obtained by removing consecutive zeros from $X_i$, we see that (in a block matrix form), 
\[
\begin{bmatrix}
A_{n - 1} & * & * & \cdots  & *& *\\
0 & A_{n - 2} & * & \cdots &* & *\\
\vdots & \vdots & \vdots &\ddots &\vdots &\vdots \\
0 & 0 & 0 & \cdots & 0 & A_1
\end{bmatrix}
\begin{bmatrix}
X_1'\\
X_2'\\
\vdots\\
X_{n-1}'
\end{bmatrix}
= \vec{b},
\]
where $\vec{b}$ is $n(n-1)/2 \times 1$ column vector consisting of strictly upper anti-diagonal entries from $B$. Note that the coefficient matrix of the above equation is nonsingular, since its determinant equals to $\det(A_{1}) \cdot \cdots \cdot \det(A_{n-1}),$ which is nonzero  by assumption. This completes the proof. 
\end{proof}

\section{Proof of Theorem \ref{thm:nonprimality}}\label{sect:proof}

\subsection{Suitable change of variables}
Let $x_{i,j}:= x_i-x_j$. We show the following is sufficient to prove Theorem \ref{thm:nonprimality}:

\begin{proposition}\label{prop:sigma(13)(24)} We have \begin{equation*}
    (x_{1,2})^{2n-1}(x_{3,4})^{2n-1} \in \langle (x_{1,3})^{2n}, (x_{1,4})^{2n}, (x_{2,3})^{2n}, (x_{2,4})^{2n}\rangle.
\end{equation*}
\end{proposition}

\begin{proof}[Proof of Theorem \ref{thm:nonprimality}, assuming Proposition \ref{prop:sigma(13)(24)}] If $\sigma$ sends $x_1 \to x_i$ and $x_2 \to x_j$,  set $x_3= x_i$, and $x_4=x_j$. Then we see that 
\begin{equation*}
    f\cdot \sigma(g) = (x_{1,2})^{2n-1}(x_{i,j})^{2n-1} \in \langle (x_{1,i})^{2n}, (x_{1,j})^{2n}, (x_{2,i})^{2n}, (x_{2,j})^{2n}\rangle  \subset I(2n).    
\end{equation*}
Hence $f \cdot \sigma(g) \in I(2n)$ for all $\sigma$.
\end{proof}

To prove the above proposition, we will make a change of variables. Set $S = k[\delta, x, y]$, and define

\begin{equation*}
    J(2n) :=\langle \delta^{2n}, (\delta+x)^{2n}, (\delta+y)^{2n},(\delta+x+y)^{2n}\rangle.
\end{equation*} 
If we write $x = x_{1,2}$, $y = -x_{3,4}$, and $\delta = x_{2,4}$, we see that 
\begin{equation*}
  (x_{1,2})^{2n-1}(x_{3,4})^{2n-1} = -(xy)^{2n-1},  
\end{equation*}
and
\begin{equation*}
    \langle (x_{1,3})^{2n}, (x_{1,4})^{2n}, (x_{2,3})^{2n}, (x_{2,4})^{2n}\rangle = J(2n).
\end{equation*}
Hence, Proposition \ref{prop:sigma(13)(24)} follows from the following lemma:
\begin{lemma}\label{lemma:changeofvariables} We have $(x y)^{2n-1} \in J(2n)$.

\end{lemma}

\subsection{The Proof} Finally, we will complete the proof of Theorem \ref{thm:nonprimality} by showing Lemma \ref{lemma:changeofvariables}. 

\begin{proof}[Proof of Lemma \ref{lemma:changeofvariables}]
We need to find homogeneous polynomials 
$\fd,\fx,\fy,\fxy \in k[\delta, x,y]$ of degree $2n-2$ such that
\begin{equation}\label{eqn:xy in I(2n)}
(xy)^{2n-1} = \fd\delta^{2n}+ \fx(\delta+x)^{2n}+ \fy(\delta+y)^{2n}+\fxy(\delta+x+y)^{2n}   
\end{equation}
Divide both sides by $\delta^{4n-2}$. We then get

\begin{equation*}
    \left(\frac{x}{\delta}\right)^{2n-1} \cdot \left( \frac{y}{\delta} \right)^{2n-1} = f_\delta'+ f_{\delta+x}'\cdot\left(1+\frac{x}{\delta}\right)^{2n} +f'_{\delta+y}\cdot\left(1+ \frac{y}{\delta}\right)^{2n} +f_{\delta+x+y}'\cdot\left(1+\frac{x}{\delta}+\frac{y}{\delta}\right)^{2n},
\end{equation*}
where $f_*'$ are polynomials of degree at most $2n-2$ obtained by dehomogenization of $f_*$'s:
\[f_*'\left(\frac{x}{\delta}, \frac{y}{\delta}  \right) = f_*\left(1, \frac{x}{\delta}, \frac{y}{\delta}\right)\]
We will find the equations between the coefficients, and although such comparison is possible without dehomogenization, it is more convenient to eliminate one of the variables. 
Then set $s = 1+x/\delta$, and $t=y/\delta$. Define $a,b,c,d$ to be polynomials in $s$ and $t$ obtained from $f_\delta', \cdots, f_{\delta+x+y}'$ after the change of variables. Then, we are reduced to show that there exist polynomials $a, b, c, d$ in $s,t$ with degrees $\leq 2n-2$ such that
\begin{equation}\label{eqn:for st}
(s-1)^{2n-1}t^{2n-1}= a + b\cdot s^{2n} + c\cdot (1+t)^{2n} + d \cdot (s+t)^{2n}.   
\end{equation}
Note that we have two gradings, obtained by dehomogenization: one from $s$--degree and another from $t$-degree. Such will allow us to obtain the relationships between the coefficients of the terms, and we will show that there exist a (unique) pair of coefficients that satisfy the relationships between the coefficients of $a, b, c, d$ that we will soon derive.

Put 
\begin{equation*}
    b = \sum_{i =0}^{2n-2}b_is^{i},
\end{equation*}
 where $b_i = b_i(t)$ is a polynomial in $t$ with degree $\leq i$ for $0 \leq i \leq 2n-2$. Similarly write $a = \sum a_is^i$, $c = \sum c_is^{i}$, and $d = \sum d_i s^{i}$. Expand the term $d(s+t)^{2n}$ in \eqref{eqn:for st}, and compare the coefficients of $s^k$: i.e., polynomial in $t$- from $k = 0$ to $k = 4n-2$. 
 
We now limit our attention to terms with $s$--degree $\geq 2n$. Observe that such terms can only appear from $b\cdot s^{2n}$ and $d\cdot (s+t)^{2n}$. As the terms from the expansion $c\cdot(1+t)^{2n}$ cannot have such terms as the degrees of $b$ (and thus $s$--degrees of those) is at most $2n-2$. Thus, we get
\begin{align*}
    d_{2n-2}t^{2n-2}{2n \choose 2n-2}+\cdots +d_{1}t{2n \choose 1}+d_{0} &= -b_{0} \text{ } (\text{for } s^{2n})\\
    &\text{  }\vdots \\
    d_{2n-2}t{2n \choose 1}+ d_{2n-3}&= -b_{2n-3} \text{ } (\text{for } s^{4n-3})\\
    d_{2n-2}&= -b_{2n-2} \text{ } (\text{for } s^{4n-2}).
\end{align*}
Multiply the first equation by $1/(s^{2n})$, thesecond by $t/s^{2n+1}$, $\cdots$, and the last equation by $t^{2n}/s^{4n-2}$. We then get the following equation:

\begin{equation}\label{relbnd}
\begin{bmatrix}
1 & {2n \choose 1} & \cdots& {2n \choose 2n-2} \\ \\
0 & 1 & \cdots & {2n \choose 2n-3}\\ \\
\vdots & \vdots&\vdots & \vdots\\ \\ 
0 & 0 &\cdots & 1
\end{bmatrix}
\cdot 
\begin{bmatrix}d_{0} \\ \\
d_{1}t \\ \\
\vdots \\ \\
d_{2n-2}t^{2n-2}
\end{bmatrix} 
= 
- 
\begin{bmatrix}
b_{0} \\ \\
b_{1}t \\ \\ 
\vdots \\ \\
b_{2n-2}t^{2n-2}
\end{bmatrix}.
\end{equation}
Thus, we see that the $d$'s uniquely determine the $b$'s.

Now, compare the terms with $s$-degree from zero to to $2n-1$ from the equation \eqref{eqn:for st}. Now the terms we have come from the remaining terms of $d\cdot(s+t)^{2n}$, from $a$ and $c\cdot(1+t)^{2n}$, and lastly from the expansion of $(s-1)^{2n-1}t^{2n-1},$ where the coefficients are entirely determined by binomial coefficients. We write the terms in that order:

\begin{align}\label{relations between c,d}
    d_{0} t^{2n} &= -a_{0}-c_{0}(1+t)^{2n}+ (-1)^{2n-1} { 2n-1 \choose 2n-1}t^{2n-1} \nonumber \\
    \sum_{i= 2n-1}^{2n}d_{i-(2n-1)}{2n \choose i}t^{i} &=-a_{1}-c_{1}(1+t)^{2n} + (-1)^{2n-2} {2n-1 \choose 2n-2} t^{2n-1}  \nonumber \\
    &\vdots\\ 
    \sum_{i=2}^{2n} d_{i-2} {2n \choose i}t^{i} &= -a_{2n-2}-c_{2n-2}(1+t)^{2n}+(-1)^{1} {2n-1 \choose 1}t^{2n-1} \nonumber \\
    \sum_{i=1}^{2n-1}d_{i-1} {2n \choose i}t^i &= -0-0+(-1)^0 {2n-1 \choose 0}t^{2n-1} \nonumber 
\end{align}
Write 
\begin{equation*}
 c_{k} = c_{k,0}+\cdots + c_{k,2n-2-k}t^{2n-2-k}, \text{ and }d_{k} = d_{k,0}+d_{k,1}t+\cdots+d_{k,2n-2-k}t^{2n-2-k}.
\end{equation*}

Here, note that $c_{k}$ and $d_{k}$ are polynomials in $t$ of degree $2n-2-k$, and $c_{i,j}$ and $d_{i,j}$ are scalar coefficients, of polynomials $c_i$ and $d_i$ respectively. Recall that $\deg a_i \leq i$, so $a_i$'s can be thought as the remainder of $(RHS) - (LHS)$ with respect to $t^i$. Hence, if we consider the quotient of $(LHS)$ and $(RHS)$ of the equation \eqref{relations between c,d} with respect to the divisor $t^i$, we get the following equations system of $2n - 2$ linear equations involving $c_{ij}$'s and $d_{ij}$'s:
\begin{equation}\label{eqn:theeqn}
    \Lambda C + D \Lambda = B.
\end{equation}
Here, $\Lambda$ is an upper triangular unipotent matrix

\begin{equation}
    \Lambda = 
    \begin{bmatrix}
1 & {2n \choose 1} & {2n \choose 2}&\cdots &{2n \choose 2n - 3} &{2n \choose 2n-2} \\ \\
0 & 1 & {2n \choose 1} & \cdots & {2n \choose 2n - 4} & {2n \choose 2n-3}\\ \\
\vdots & \vdots &\vdots & \vdots &\vdots & \vdots\\ \\
0 & 0 & 0 &\cdots & 1& {2n \choose 1} \\ \\ 
0 & 0 & 0 &\cdots & 0 & 1
\end{bmatrix},
\end{equation}
and $B$ is a diagonal matrix
\begin{equation}
    B = \textbf{Diag} \left[(-1)^{2n - 1} {2n - 1 \choose 2n - 1}, (-1)^{2n - 2} {2n - 1 \choose 2n - 2}, \cdots, (-1)^{0} {2n - 1 \choose 0}\right],
\end{equation}
and $C$, $D$ are the following strictly lower trianglular matrices
\begin{equation}
    C, D = 
    \begin{bmatrix}
0 & 0 & 0 & \cdots& 0 \\ \\
c_{0, 0} & 0 & 0 &  \cdots & 0\\ \\
c_{0, 1} & c_{1,0} & 0 & \cdots & 0\\ \\
\vdots & \vdots&\vdots &\vdots & \vdots\\ \\ 
c_{0, 2n-2} & c_{1, 2n-3} & c_{2, 2n - 4} &\cdots & c_{2n-2, 0}
\end{bmatrix},
    \begin{bmatrix}
0 & 0 & 0 & \cdots& 0 \\ \\
d_{0, 0} & 0 & 0 &  \cdots & 0\\ \\
d_{0, 1} & d_{1,0} & 0 & \cdots & 0\\ \\
\vdots & \vdots&\vdots &\vdots & \vdots\\ \\ 
d_{0, 2n-2} & d_{1, 2n-3} & d_{2, 2n - 4} &\cdots & d_{2n-2, 0}
\end{bmatrix}.
\end{equation}

Using Proposition \ref{prop:minormatrix}, we see that $\Lambda, B$ satisfies the condition of Lemma \ref{anti-diagonal}. Hence by Lemma \ref{anti-diagonal}, there is a pair $(C, D)$ with $D = - C^{\tau}$ that satisfies \eqref{eqn:theeqn}. This completes the proof.
\end{proof}

\begin{remark} Conceptually, the reason why we expect an equation symmetric in $C$ (the coefficient matrix of $c$) and $D$(those of $d$) is that once we apply different change of variables and dehomogenizations to \eqref{eqn:xy in I(2n)}, we can actually interchange $c$ and $d$. Consider the (well-defined linear) map $\phi$ induced by 
\begin{align}
\delta &\mapsto \delta + x, \\
\delta + x &\mapsto \delta, \\
\delta + y &\mapsto \delta + x + y, \\
\delta + x+ y &\mapsto \delta + y,
\end{align}
and dehomogenize after applying $\phi$. Then, $c$ and $d$ are mapped to $d$ and $c$. As a consequence we expect to observe some symmetry, i.e., one that we see from \eqref{eqn:theeqn}.

\end{remark}

\begin{example}
By unraveling, we can compute explicit polynomials $a,b,c,d$ satisfying $(xy)^{2n-1} = a \delta^{2n} + \cdots + d (\delta+x+y)^{2n}$ for any specific values of $n$. For instance, when $2n=4$, we have 
\begin{center}
$\begin{bmatrix}
a(\delta,x,y)\\
\\
b(\delta,x,y)\\
\\
c(\delta,x,y)\\
\\
d(\delta,x,y)
\end{bmatrix}=
\begin{bmatrix}
\frac{1}{5} \delta^2+ \frac{3}{5}\delta x + \frac{3}{5} \delta y + \frac{3}{2}xy\\
\\
-\frac{1}{5}\delta^2 + \frac{1}{5}\delta x + \frac{2}{5}x^2-\frac{3}{5}\delta y + \frac{9}{10}xy\\
\\
-\frac{1}{5}\delta^2 - \frac{3}{5}\delta x +\frac{1}{5}\delta y + \frac{9}{10}xy + \frac{2}{5}y^2\\
\\
\frac{1}{5}\delta^2 -\frac{1}{5}\delta x - \frac{2}{5}x^2 - \frac{1}{5}\delta y +\frac{7}{10}xy -\frac{2}{5}y^2
\end{bmatrix},
$
\end{center}
and we can check that indeed,
\begin{equation*}
  a\delta^{4}+b(\delta+x)^4+c(\delta+y)^4+d(\delta+x+y)^4 = (xy)^3.  
\end{equation*}

\end{example}

\begin{remark} We can replace $k$, a field of characteristic zero, with any field of characteristic $p$ such that the statement of Theorem \ref{prop:minormatrix} holds.
\end{remark}

In fact, we can show the following general result:
\begin{theorem}\label{theorem:generalization} Let $N$ be a positive integer. Then the contracted ideal $\fJ(N) := J(N) \cap k[x,y]$ contains $\langle x^{2N-1}, y^{2N-1}, x^{2i+1}y^{2j+1}| i+j = N-2 \rangle$. 
\end{theorem}
\begin{proof} Without loss of generality assume $i \leq j$. We can show that $x^{2i+1}y^{2j+1} \in J(N)$ with the same dehomogenization and comparison of coefficients, and we see that
\begin{equation*}
    x^{2N-1} = ((\delta+x)-\delta)^{2N-1} \in \langle \delta^N, (\delta+x)^N\rangle
\end{equation*} from binomial expansion.
\end{proof}

Note that Theorem \ref{thm:nonprimality} is a special case of Theorem \ref{theorem:generalization}. Furthermore, we have a following conjecture that the contracted ideal $\fJ(N)$ is actually a monomial ideal:

\begin{conjecture}\label{Conj:contracted ideal} In fact, $\fJ(N)$ equals to $\langle x^{2N-1}, y^{2N-1}, x^{2i+1}y^{2j+1}| i+j = N-2 \rangle$.
\end{conjecture}

\section{Open Problems}\label{sect:openproblems}
In this section, we give more conjectures that generalize Theorem \ref{thm:nonprimality} and Theorem \ref{theorem:generalization}.
\subsection{More variables} Fix  $r, n \in \N$, and for a subset $S$ of $[r] := \{1,2, \cdots, r\}$, define $t_S = \sum_{i \in S} t_i$, and $t_{S,0} = t_0 + t_S$. Here, $t_0$ is a special variable, that takes a role of $\delta$ in \eqref{eqn:xy in I(2n)}. We then define $\ird \subset R := K[t_0, \cdots, t_r]$ to be  
\begin{equation*}
  \ird:= \langle (t_{S,0})^d: S \subset [r]\rangle.  
\end{equation*}

\begin{conjecture}\label{Conj:more variables} For any $r\geq 2$ and $n\geq 1$, we have 
\begin{equation*}
  \left( t_1\cdots t_r\right)^{2n-1} \in I_{r}^{(nr)}.  
\end{equation*}
\end{conjecture}
For $r =1$, the above conjecture reduces to $t_1^{2n-1} \in \langle t_0^n, (t_0+t_1)^n \rangle,$ which follows easily from the binomial expansion of $t^{2n-1} = ((t+\delta)-\delta)^{2n-1}.$ Our theorem gives a positive answer for $r=2$. We have verified the conjecture for some small values of $r \ge 3$ by computer.

\subsection{Betti tables}
We also observe some patterns from Betti tables of the family of ideals $\ird$ with fixed $r$. Let $\mathbf{F}$ be the minimal free resolution of $\ird$, and write 
\[F_i = \bigoplus_{j \in \mathbf{Z}} R(-j)^{\beta_{i,j}}.\]

When $r = 1$, we expect $\beta_{i,j} = 0$ for all $i,j$ except for $(i,j) = (0,0), (d-1, 1),$ and $(2d - 2, 2)$. The nonzero values of $\beta_{i,j}$'s are given by $\beta_{0,0} = \beta_{2d - 2, 2} = 1, \beta_{d - 1, 1} = 2$. See the following Betti table, where all the omitted rows are zero rows, and entries with $\cdot$ is a zero entry:
\begin{center}
\begin{tikzpicture}
\matrix (m) [matrix of nodes,
nodes in empty cells,
row sep=-\pgflinewidth,
column sep=-\pgflinewidth,
nodes={minimum height=5mm,minimum width=16mm,anchor=center},
row 1/.style={nodes={minimum height=9mm,font=\bfseries}},
column 1/.style={nodes={minimum width=16mm}},
]{
      & 0      &  1    & 2     \\ 
0     & 1      &  $\cdot$    & $\cdot$   \\
\vdots&\vdots  & \vdots& \vdots  \\
$d - 1$ &  $\cdot$   & 2   & $\cdot$  \\
\vdots&\vdots  & \vdots& \vdots  \\
$2d - 2$   & $\cdot$     & $\cdot$    & 1  \\
};
\draw[teal] 
(m-1-1.north west)--(m-1-1.south east)
(m-1-1.north east)--(m-6-1.south east)
(m-1-1.north west)--(m-1-4.north east)
(m-1-1.south west)--(m-1-4.south east)
(m-2-1.south west)--(m-2-4.south east)
(m-3-1.south west)--(m-3-4.south east)
(m-4-1.south west)--(m-4-4.south east)
(m-5-1.south west)--(m-5-4.south east)
(m-6-1.south west)--(m-6-4.south east)
(m-1-1.south west) node[above right]{\bfseries \textit{i}}
(m-1-1.north east) node[below left]{\bfseries \textit{j}};
\end{tikzpicture}
\end{center}

For For $r = 2$, we have the following conjectured Betti table obtained for some small values we tested by computer:
\begin{center}
\begin{tikzpicture}
\matrix (m) [matrix of nodes,
nodes in empty cells,
row sep=-\pgflinewidth,
column sep=-\pgflinewidth,
nodes={minimum height=5mm,minimum width=16mm,anchor=center},
row 1/.style={nodes={minimum height=9mm,font=\bfseries}},
column 1/.style={nodes={minimum width=16mm}},
]{
      & 0      &  1    & 2     & 3    \\ 
0     & 1      &  $\cdot$    & $\cdot$     & $\cdot$  \\
\vdots&\vdots  & \vdots& \vdots& \vdots  \\
$ d- 1$&  $\cdot$    & 4     & $\cdot$ & $\cdot$ \\
\vdots&\vdots  & \vdots& \vdots& \vdots \\
$2d - 3$& $\cdot$     &$\cdot$      & $d$   & $\cdot$  \\
$2d - 2$& $\cdot$     &$\cdot$      & 3& $d$ \\ 
};
\draw[teal] 
(m-1-1.north west)--(m-1-1.south east)
(m-1-1.north east)--(m-7-1.south east) 
(m-1-1.north west)--(m-1-5.north east)
(m-1-1.south west)--(m-1-5.south east)
(m-2-1.south west)--(m-2-5.south east)
(m-3-1.south west)--(m-3-5.south east)
(m-4-1.south west)--(m-4-5.south east)
(m-5-1.south west)--(m-5-5.south east)
(m-6-1.south west)--(m-6-5.south east)
(m-7-1.south west)--(m-7-5.south east)
(m-1-1.south west) node[above right]{\bfseries \textit{i}}
(m-1-1.north east) node[below left]{\bfseries \textit{j}};
\end{tikzpicture}
\end{center}

Lastly, we give the Betti table for $r = 3$:

\begin{center}
\begin{tikzpicture}
\matrix (m) [matrix of nodes,
nodes in empty cells,
row sep=-\pgflinewidth,
column sep=-\pgflinewidth,
nodes={minimum height=5mm,minimum width=16mm,anchor=center},
row 1/.style={nodes={minimum height=9mm,font=\bfseries}},
column 1/.style={nodes={minimum width=16mm}},
]{
      & 0      &  1    & 2     & 3    & 4\\ 
0     & 1      &  $\cdot$    & $\cdot$     & $\cdot$    & $\cdot$\\
\vdots&\vdots  & \vdots& \vdots& \vdots &\vdots \\
$d - 1$&  $\cdot$    & 8     & $\cdot$ & $\cdot$  & $\cdot$\\
\vdots&\vdots  & \vdots& \vdots& \vdots & \vdots \\
$2d - 4$& $\cdot$     &$\cdot$ & $(d - 1)d/2$ & $\cdot$ & $\cdot$ \\
$2d - 3$& $\cdot$     &$\cdot$ & $4d$  & $(d- 1)(d+1)$ & $\cdot$ \\
$2d - 2$& $\cdot$     &$\cdot$ & 6& $4d$ & $d(d-1)/2$ \\ 
};
\draw[teal] 
(m-1-1.north west)--(m-1-1.south east)
(m-1-1.north east)--(m-8-1.south east) 
(m-1-1.north west)--(m-1-6.north east)
(m-1-1.south west)--(m-1-6.south east)
(m-2-1.south west)--(m-2-6.south east)
(m-3-1.south west)--(m-3-6.south east)
(m-4-1.south west)--(m-4-6.south east)
(m-5-1.south west)--(m-5-6.south east)
(m-6-1.south west)--(m-6-6.south east)
(m-7-1.south west)--(m-7-6.south east)
(m-8-1.south west)--(m-8-6.south east)
(m-1-1.south west) node[above right]{\bfseries \textit{i}}
(m-1-1.north east) node[below left]{\bfseries \textit{j}};
\end{tikzpicture}
\end{center}

There are a number of patterns we can observe: why do we expect each entries of Betti table to stabilize? For instance, why do we expect to get ``6'' at the lower left vertex of the triangle (i.e., $\beta_{2d-2, 3} = 6$ for $r=3$)? More fundamentally, why do we see $r \times r$ right triangle at the lower right corner? Should we expect to observe similar patterns for larger values of $r$?

\end{document}